\documentclass[12pt]{article}
\usepackage{amsmath,amssymb,amsthm}
\usepackage{graphicx}
\usepackage{subfigure}
\usepackage{amsmath,amsfonts,amssymb,amscd}
\usepackage{indentfirst,graphicx,epstopdf}
\usepackage{graphicx,psfrag,ifpdf,enumerate}
\usepackage{caption}
\usepackage{amsthm, amsfonts, color}
\usepackage[bookmarksnumbered, plainpages]{hyperref}
\usepackage[numbers,sort&compress]{natbib}
\usepackage[noend]{algpseudocode}
\usepackage{algorithmicx,algorithm}
\usepackage{psfrag}
\usepackage{color}
\usepackage{enumerate}
\usepackage{epstopdf}
\input{epsf}

\newtheorem{theorem}{Theorem}

\newtheorem{lem}[theorem]{Lemma}%[section]

\newtheorem{cor}[theorem]{Corollary}%[section]
\newtheorem{prop}[theorem]{Proposition}%[section]
\theoremstyle{definition}
\newtheorem{Def}[theorem]{Definition}%[section]

\newtheorem{fact}[theorem]{Fact}

\title{\Large \bf Rainbow triangles sharing one common vertex or edge\footnote{BN was partially supported by
NSFC (No. 11971346).}}
\author{{\small Xiaozheng Chen$^a$ and Bo Ning$^b$}\footnote{Corresponding author.}\\
{\small $^a$School of Mathematics and Statistics}\\
{\small Zhengzhou University, Zhengzhou 450000, China}\\
{\small Email: cxz@zzu.edu.cn}\\
{\small $^b$College of Computer Science}\\
{\small Nankai University, Tianjin 300350, China}\\
{\small Email: bo.ning@nankai.edu.cn}}

\date{}
\begin{document}
\maketitle

\begin{abstract}
Let $G$ be an edge-colored graph on $n$ vertices. For a vertex
$v$, the \emph{color degree} of $v$ in $G$, denoted by $d^c(v)$,
is the number of colors appearing on the edges incident with $v$.
Denote by $\delta^c(G)=\min\{d^c(v):v\in V(G)\}$.
By a theorem of H. Li, an $n$-vertex edge-colored graph $G$ contains a rainbow
triangle if $\delta^c(G)\geq \frac{n+1}{2}$. Inspired by this
result, we consider two related questions concerning
edge-colored books and friendship subgraphs of edge-colored graphs.
Let $k\geq 2$ be a positive integer.
We prove that if $\delta^c(G)\geq \frac{n+k-1}{2}$ where $n\geq 3k-2$,
then $G$ contains $k$ rainbow triangles sharing one common edge;
and if $\delta^c(G)\geq \frac{n+2k-3}{2}$ where $n\geq 2k+9$,
then $G$ contains $k$ rainbow triangles sharing one common vertex.
The special case $k=2$ of both results improves H. Li's theorem.
The main novelty of our proof of the first result is a combination of the recent
new technique for finding rainbow cycles due to Czygrinow, Molla, Nagle, and Oursler
and some recent counting technique from \cite{LNSZ}. The proof of the second result
is with the aid of the machine implicitly in the work of Tur\'an numbers
for matching numbers due to Erd\H{o}s and Gallai.
\\[2mm]
{\bf AMS Classification 2020:} 05C15, 05C38.

\end{abstract}

\section{Introduction}
In 1907, Mantel \cite{M07} proved that every triangle-free graph on $n$ vertices has size at
most $\lfloor\frac{n^2}{4}\rfloor$. Rademacher (see \cite[pp.91]{EFGG})
showed that there are indeed at
least $\lfloor\frac{n}{2}\rfloor$ triangles in a graph $G$ on $n$ vertices
and at least $\frac{n^2}{4}+1$ edges but not only one triangle.
The $k$-fan (usually called \emph{friendship graph}), denoted by $F_k$, is a graph
which consists of $k$ triangles sharing a common vertex.
The \emph{book} $B_k$ is a graph which consists of $k$ triangles sharing a common edge.
Erd\H{o}s \cite{Er} extended Mantel's theorem and conjectured that
there is a $B_{\lceil\frac{n}{6}\rceil}$ in $G$ if $e(G)>\frac{n^2}{4}$, which was
later confirmed by Edwards in an unpublished
manuscript \cite{Ed} and independently by Khad\v{z}iivanov and Nikiforov \cite{KN}.
Erd\H{o}s, F\"{u}redi, Gould, and Gunderson \cite{EFGG} also
studied Tur\'an numbers of $F_k$,
and proved that $ex(n,F_k)=\lfloor \frac{n^2}{4}\rfloor+k^2-k$
if $k$ is odd; and $ex(n,F_k)=\lfloor \frac{n^2}{4}\rfloor+
k^2-\frac{3k}{2}$ if $k$ is even, for $n\geq 50k^2$. These results
immediately imply the fact that every graph on $n$ vertices
with minimum degree at least $\frac{n+1}{2}$ contains a $B_k$
for $n\geq 6k$ and also a $F_k$ for $n\geq 50k^2$.
In this paper, we consider edge-colored versions of these extremal problems.

A subgraph of an edge-colored graph is
\emph{properly colored} (\emph{rainbow})
if every two adjacent edges (all edges) have
pairwise different colors. The rainbow
and properly-colored subgraphs have been shown closely related
to many graph properties and other topics, such
as classical stability results on Tur\'an functions \cite{Y21},
Bermond-Thomassen Conjecture \cite{FLW19},
and Caccetta-H\"{a}ggkvist Conjecture \cite{ADH19},
etc. For more rainbow and properly-colored
subgraphs under Dirac-type color degree conditions,
we refer to \cite{FLZ18,FNXZ19,EM20,CMNO21,DHWY22}.

The study of rainbow triangles has a rich history,
and there are many classical open problems
on them. In some classical problems, the host graph is complete.
One conjecture due to Erd\H{o}s and S\'{o}s \cite{EH}
asserts that the maximum number of rainbow triangles in
a 3-edge-coloring of the complete graph $K_n$, denoted by $F(n)$,
satisfied that
$F(n)=F(a)+F(b)+F(c)+F(d)+abc+abd+acd+bcd$, where $a+b+c+d=n$
and $a,b,c,d$ are as equal as possible. By using Flag Algebra,
Balogh et al. \cite{B} confirmed this conjecture for $n$ sufficiently large and $n=4^{k}$
for any $k\geq 1$. The other example is a recent conjecture
by  Aharoni (see \cite{ADH19}), which can imply Caccetta-H\"{a}ggkvist
Conjecture \cite{CH}, a big and fundamental open problem in digraph.
The conjecture says that given any positive integer $r$,
if $G$ is an $n$-vertex edge-colored graph with $n$ color classes,
each of size at least $n/r$, then $G$ contains a rainbow cycle of length at most $r$.
For more recent developments on Aharoni's conjecture, we refer
to the work \cite{DDF,HPPS,HS} and
more references therein. A special case of Aharoni's
conjecture is about rainbow triangles. The relationship between
directed triangles and rainbow triangles has been extensively
used before (see \cite{LW12,L13,LNXZ14}).

We would like to introduce a
construction from Li \cite{L13}. Suppose that $D$ is an $n$-vertex digraph satisfying
out-degree of every vertex is at least $n/3$. Let $V(D)=\{v_1,v_2,\ldots,v_n\}$.
We construct an edge-colored graph $G$ such that: $V(G)=V(D)$; for each
arc $v_iv_j\in A(D)$, we color the edge $v_iv_j$ with the color $j$.
In this way, the number of colors appearing on edges incident with $v_i$
different from $i$ equals to $d^+_D(v_i)$. Thus, finding a directed triangle
in $D$ is equivalent to finding a rainbow triangle in such an
edge-colored graph. More importantly for us, the idea of constructing
the auxiliary digraph will also be an important constitution for our proofs in
this paper.

Our theme of this paper is closely related
to the color degree conditions for rainbow triangles.
Let $G$ be an edge-colored graph.
For every vertex $v\in V(G)$, the \emph{color degree} of $v$,
denoted by $d^c_G(v)$, is the number of distinct colors appearing
on the edges which are incident to $v$.
The \emph{minimum color degree} of $G$, denoted by $\delta^c(G)$
(or in short, $\delta^c$), equals to $\min\{d^c(v):v\in V(G)\}$.
It is an easy observation
that every graph on $n$ vertices contains a triangle
if minimum degree is at least $\frac{n+1}{2}$.
H. Li and Wang \cite{LW12} considered a rainbow version and conjectured
that the minimum color degree condition $\delta^c(G)\geq \frac{n+1}{2}$
ensures the existence of a rainbow triangle in $G$.
This conjecture was confirmed by H. Li \cite{L13} in 2013.
\begin{theorem}[H. Li \cite{L13}]\label{Thm:Li}
Let $G$ be an edge-colored graph on $n$ vertices. If $\delta^c(G)\geq \frac{n+1}{2}$
then $G$ contains a rainbow triangle.
\end{theorem}

Independently, B. Li, Ning, Xu and Zhang \cite{LNXZ14}
proved a weaker condition $\sum_{v\in V(G)}d^c(v)\geq \frac{n(n+1)}{2}$
suffices for the existence of rainbow triangles, and moreover, characterized
the exceptional graphs under the condition $\delta^c(G)\geq \frac{n}{2}$.
Very recently, X. Li, Ning, Shi and Zhang \cite{LNSZ} proved
a counting version of Theorem \ref{Thm:Li}, i.e.,
there are at least $\frac{1}{6}\delta^c(G)(2\delta^c(G)-n)n$
rainbow triangles in an edge-colored graph $G$, which is best possible.

Hu, Li and Yang \cite{HLY20} proposed the following conjecture:
Let $G$ be an edge-colored graph on $n\geq 3k$ vertices. If
$\delta^c\geq \frac{n+k}{2}$ then $G$ contains $k$ vertex-disjoint
rainbow triangles.
Besides the work on Tur\'an numbers of books and $k$-fans mentioned before,
our another motivation is to study the converse of Hu-Li-Yang's conjecture,
i.e., rainbow triangles sharing vertices or edges. As the selections for us,
we shall study the existence of rainbow triangles sharing one common vertex
or an edge under color degree conditions, in views of the famous
books and friendship graphs in graph theory.

Our original result is the following one which improves Li's theorem. Indeed, we can go farther.
\begin{theorem}\label{Thm:Original}
Let $G$ be an edge-colored graph on $n$ vertices with $\delta^c(G)\geq \frac{n+1}{2}$.\\
(i) If $n\geq 5$ then $G$ contains two rainbow triangles sharing one common edge;\\
(ii) If $n\geq 13$ then $G$ contains two rainbow triangles sharing one common vertex.
\end{theorem}

This paper is organized as follows. In Section \ref{Sec:2}, we will
list our main theorems.
In Section \ref{Sec:3}, we introduce some necessary notations
and terminology and prove some lemmas. In Section \ref{Sec:4},
we prove general versions of Theorem \ref{Thm:Original}, i.e.,
Theorems \ref{Thm1:Bk} and \ref{Thm2:Fk}.
We conclude this paper with some more discussions on the sharpness of
our results, together some propositions on $F_k$ and $B_k$
on uncolored graphs.

\section{Main theorems}\label{Sec:2}
Our main results are given as follows.
\begin{theorem}\label{Thm1:Bk}
Let $k\geq 2$ be a positive integer
and $G$ be an edge-colored graph on $n\geq 3k-2$ vertices.
If $\delta^c(G)\geq \frac{n+k-1}{2}$
then $G$ contains $k$ rainbow triangles sharing one edge.
\end{theorem}

\begin{theorem}\label{Thm2:Fk}
Let $k\geq 2$ be a positive integer and
$G$ be an edge-colored graph on $n\geq 2k+9$ vertices.
If $\delta^c(G)\geq \frac{n+2k-3}{2}$
then $G$ contains $k$ rainbow triangles only sharing one common
vertex.

\end{theorem}
Setting $\delta^c(G)= \frac{n+k-1}{2}$ in Theorem \ref{Thm1:Bk}, the following
example shows that the bound ``$n\geq 3k-2$" is sharp.
Furthermore, it follows from Example 1 that
the tight color degree should be at least $\delta^c\geq \frac{n+k}{2}$
when $n\leq3k-3$.

\vspace{2mm}

\noindent{\bf Example 1.}
Let $G$ be a properly-colored balanced complete 3-partite graph $G[V_1, V_2, V_3]$
with $|V(G)|=3k-3$ and $|V_1|=|V_2|=|V_3|=k-1$, where $k\geq 1$ is a
positive integer. Then for each vertex $v\in V(G)$,
$d^c(v)=d(v)=2k-2=\frac{n+k-1}{2}$
while $G$ contains no $B_k$ and $F_k$ (see Figure \ref{fig1}).

\begin{figure}[htbp!]
  \centering
\scalebox{1}{\includegraphics[width=2.3in,height=2.0in]{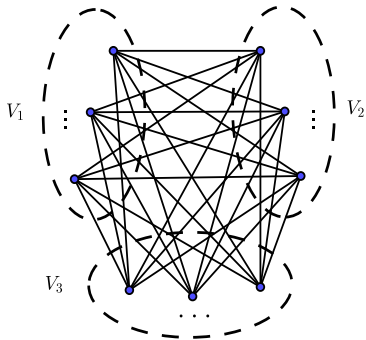}}\\
\captionsetup{font={scriptsize}}
\caption{An extremal graph for Theorem \ref{Thm1:Bk}}
\end{figure}\label{fig1}
The main novelty of our proof of Theorem \ref{Thm1:Bk} is a combination of the recent
new technique for finding rainbow cycles due to Czygrinow, Molla, Nagle, and Oursler \cite{CMNO21}
and some recent counting technique from \cite{LNSZ}. In particular, Czygrinow et al. \cite{CMNO21} extended
H. Li's theorem by proving that for every integer $\ell$, every edge-colored graph $G$
on $n\geq 432\ell$ many vertices satisfying $\delta^c(G)\geq\frac{n+1}{2}$ admits a
rainbow $\ell$-cycle $C_{\ell}$. One novel concept introduced in \cite{CMNO21} is called restriction
color, and it will be used in our proof.
The proof of Theorem \ref{Thm2:Fk}
is with the aid of the machine implicitly in the work of Tur\'an number
for matching numbers due to Erd\H{o}s and Gallai \cite{EG59}. (For details,
see the proof of Lemma \ref{Lem:alpha}.)

Meantime, both Theorem \ref{Thm1:Bk} and Theorem \ref{Thm2:Fk} improve Theorem \ref{Thm:Li} (by setting $k=2$).
On the other hand, Theorem \ref{Thm2:Fk} slightly improves Theorem 9 in \cite{LNSZ} by a different method.

\vspace{3mm}

\section{Additional notations and lemmas}\label{Sec:3}
Some of our notations come from \cite{CMNO21,LNSZ}. Let $G$
be an edge-colored graph. Let $\mathcal{C}: E(G)\rightarrow \{1,2,\ldots,k\}$
be an edge-coloring of $G$. For a color $\alpha\in\mathcal{C}(G)$ and a vertex $v\in V(G)$,
define the $\alpha$-\emph{neighborhood} of $v$ as
$$N_{\alpha}(v)=\{u\in N(v)\mid c(uv)=\alpha\},$$
$\alpha$-\emph{neighborhood in} $X$ of $v$ as
$$N_{\alpha}(v,X)=\{u\in N(v)\cap X\mid c(uv)=\alpha\}$$
where $X\subseteq V(G)$, and $N(v)$ is the neighborhood of $v$ in $G$.
Define $$N_!(v)=\bigcup_{\alpha\in \mathcal{C}(G)}\{N_{\alpha}(v) \mid |N_{\alpha}(v)|=1\}.$$
For the sake of simplicity,
let $d_\alpha(v)=|N_\alpha(v)|$ and $d_\alpha(v,X)=|N_{\alpha}(v,X)|$.
Moreover, let $N[v]=N(v)\cup\{v\}$.
The \emph{monochromatic degree of $v$}, denoted by $d^{mon}(v)$,
is the maximum number of edges incident to $v$ colored with a same color
(i.e., $d^{mon}(v)=\text{max}\{d_{\alpha}(v) \mid \alpha\in \mathcal{C}(G)\}$.)
For a graph $G$, we denote the \emph{monochromatic degree} of $G$
by $\Delta^{mon}(G)=\text{max}\{d^{mon}(v)\mid v\in V(G)\}$.

W.l.o.g., assume that $d_1(v)\geq d_2(v)\geq \cdots\geq d_s(v)$ and $s:=s(v)=d^c(v)$
for each vertex $v\in V(G)$. (When there is no danger of ambiguity, we use
$s$ for short.)

The following concept of restriction was firstly
introduced in \cite[Section~3]{CMNO21}.
\begin{Def}[restriction color \cite{CMNO21}]
Let $G$ be an edge-colored graph.
Fix $v\in V(G)$ and $X\subseteq N(v)$.
Suppose $\alpha=c(xy)$ for $x\in X\cap N(y)$ and $y\in V(G)\setminus\{v\}$.
We say $(v,X)$ restricts color $\alpha$ for $y$ if $vxy$ is a rainbow $P_3$ and $\alpha \notin \mathcal{C}(y,N(y)\setminus X)$.
Denote by $\sigma_{v,X}(y)$
the number of colors $\alpha\in \mathcal{C}(E)$ restricted for $y$ by $(v,X)$.
\end{Def}

Denote by $rt(v)$  the number of
rainbow triangles containing $v$;
by $rt(v,x)$ the number of
rainbow triangles containing both $v$ and $x$ (i.e.,
containing the edge $vx$);
and $rt(v,X)=\sum_{x\in X}rt(v,x)$.

According to the definition of restriction color,
we have the following proposition.

\begin{prop}
Let $v$ be a vertex of $G$ and $x\in N(v)$.
Then $rt(v,x)\geq \sigma_{v,N(v)\setminus N_{c(vx)}(v)}(x)$
and $rt(v,N_{c(vx)}(v))\geq\sum_{x\in N_{c(vx)}(v)}\sigma_{v,N(v)\setminus N_{c(vx)}(v)}(x)$.
\end{prop}

\begin{proof}
For $y\in N(v)\setminus N_{c(vx)}(v)$, we have $c(vx)\neq c(vy)$.
If $c(xy)$ is a restriction color for $x$ to $v$
with respect to $N(v)\setminus N_{c(vx)}(v)$,
then $c(vy)\neq c(xy)$ and $c(xy)\notin \mathcal{C}(x,N(x)\setminus (N(v)\setminus N_{c(vx)}(x)))$,
and $c(xy)\neq c(xv)$. Thus,
$vxyv$ is a rainbow triangle.
It follows $rt(v,x)\geq \sigma_{v,N(v)\setminus N_{c(vx)}(v)}(x)$
and $rt(v,N_{c(vx)}(v))\geq\sum_{x\in N_{c(vx)}(v)}\sigma_{v,N(v)\setminus N_{c(vx)}(v)}(x)$.
\end{proof}

\vspace{4mm}
\noindent{\bf Remark 1.} Throughout this paper,
we say that an edge-colored graph $G$ is \emph{edge-minimal} if
for any $e\in E(G)$, there exists a vertex $v$ incident to $e$ such that
$d^c_{G-e}(v)<d^c_G(v)$.
Hence, $G$ contains neither monochromatic paths of length
3 nor monochromatic triangles.

The form of the following lemma is motivated by \cite{LNSZ},
but its proof is a mixture of techniques from \cite{CMNO21,LNSZ}.
For the sake of simplicity, define the color set $\mathcal{C}(G)=\{1,2,\cdots,c(G)\}$.

\begin{lem}\label{lem:rt(v,N_i(v))}
Let $G$ be an edge-colored graph which is edge-minimal.
Then for $1\leq i\leq s=d^c(v)$,
we have,
$$\begin{aligned}
rt(v,N_i(v))
\geq& \sum_{x\in N_i(v)}\Big(d^c(x)+d^c(v)-n\Big)+
d_i(v)\sum_{1\leq j\leq s}\Big(d_j(v)-1\Big)\\
&-d_i(v)\big(d_i(v)-1\big)-\sum_{y\in N_!(v)}d_{c(vy)}(y,N_i(v)).
\end{aligned}$$
\end{lem}

\noindent{\bf Proof of Lemma \ref{lem:rt(v,N_i(v))}.}
For convenience, let $X_i=N_i(v)$, $Y_i=N(v)\setminus X_i$ for $1\leq i\leq s$,
and $d_i$ for $d_i(v)$ in the following. Let $s=d^c(v)$.
For $v\in V(G)$ and $1\leq i\leq s$, define a directed graph $D_i$
on $V(D_i)=X_i\cup Y_i$
as follows: the arc $\overrightarrow{yx}$ exists if and only if the following hold:
(1) $x\in X_i$ and $y\in Y_i$;
and (2) $c(xy)=c(vy)$. Since $G$ is edge-minimal, the existence of
$\overleftarrow{xy}$ gives $d_{c(vy)}(v)=1$. (Indeed, since the arc
$\overleftarrow{xy}$ exists, we have $c(xy)=c(vy)$. If
$d_{c(vy)}(v)\geq 2$, there exists a monochromatic path of length 3,
a contradiction!)
Hence $y\in N_!(v)$. Evidently, $d^+_D(y)\leq d_{c(vy)}(y,X_i)$.
Thus,
\begin{gather}\label{e:sum-d^-_D(x)}
\sum_{x\in X_i}d^-_D(x)=\sum_{y\in Y_i}d^+_D(y)
=\sum_{y\in N_!(v)}d^+_D(y)
\leq \sum_{y\in N_!(v)}d_{c(vy)}(y,X_i).
\end{gather}
As for $x\in X_i$, there are
at most $d^-_D(x)+\sigma_{v,Y_i}(x)$ colors
which only appear on edges from $x$ to $Y_i$.
Then there are at least $d^c(x)-d^-_D(x)-\sigma_{v,Y_i}(x)$
vertices in $V(G)\setminus (Y_i\cup \{x\})$.
Therefore,
$$\begin{aligned}
n-|Y_i|-1\geq d^c(x)-d^-_D(x)-\sigma_{v,Y_i}(x)
\end{aligned}$$
\begin{gather}\label{e:d^-_D(x)}
\Rightarrow~~~\sigma_{v,Y_i}(x)\geq d^c(x)+|Y_i|+1-d^-_D(x)-n
\end{gather}

Note that $|Y_i|=d(v)-d_i=\sum_{1\leq j\leq s}d_j-d_i=d^c(v)+\sum_{1\leq j\leq s}(d_j-1)-d_i$.
Combining Ineqs (\ref{e:sum-d^-_D(x)}) and (\ref{e:d^-_D(x)}), we can get that
$$\begin{aligned}
&\sum_{x\in X_i}\sigma_{v,Y_i}(x)\\
\geq& \sum_{x\in X_i}\Big(d^c(x)+|Y_i|+1-n\Big)
-\sum_{y\in N_!(v)}d_{c(vy)}(y,X_i)\\
\geq& \sum_{x\in X_i}\Big(d^c(x)+d^c(v)+\sum_{1\leq j\leq s}(d_j-1)-d_i+1-n\Big)
-\sum_{y\in N_!(v)}d_{c(vy)}(y,X_i)\\
\geq& \sum_{x\in X_i}\Big(d^c(x)+d^c(v)-n\Big)+
d_i\Big(\sum_{1\leq j\leq s}(d_j-1)\Big)-d_i\big(d_i-1\big)-\sum_{y\in N_!(v)}d_{c(vy)}(y,X_i).
\end{aligned}$$

The proof is complete.
$\hfill\blacksquare$

\vspace{4mm}
Since $rt(v)\geq \frac{1}{2}\sum_{1\leq i\leq d^c(v)}rt(v,N_i(v))$,
we have the following corollary.

\begin{cor}\label{Cor:rt(v)}
Let $G$ be an edge-colored graph which is edge-minimal.
Then for $1\leq i\leq d^c(v)=s$,
we have,
$$
\begin{aligned}
rt(v)
\geq& ~\frac{1}{2}~\Big(
\sum_{x\in N(v)}\big(d^c(x)+d^c(v)-n\big)
+d(v)\big(\sum_{1\leq j\leq s}(d_j(v)-1)\big)\\
&-\sum_{1\leq i\leq s}d_i(v)\big(d_i(v)-1\big)-\sum_{y\in N_!(v)}d_{c(vy)}(y,N(v))\Big).
\end{aligned}
$$
$\hfill\blacksquare$
\end{cor}

\section{Edge-colored books and friendship graphs}\label{Sec:4}

\subsection{Edge-colored books}
The aim of this subsection is to prove Theorem \ref{Thm1:Bk}.

Before the proof, it's better to introduce a notation,
which indeed is some term from Corollary \ref{Cor:rt(v)}.
For $i\in [1,s]$,
let
\begin{equation}\label{e:B_i(v)}
B_i(v)=d_i(v)\Big(\sum_{1\leq j\leq s}\big(d_j(v)-1\big)\Big)-d_i(v)\big(d_i(v)-1\big)-\sum_{y\in N_!(v)}d_{c(vy)}(y,N_i(v)).
\end{equation}
Set
\begin{equation}\label{e:B}
\begin{aligned}
B(v)&=\sum_{1\leq i\leq s}B_i(v)\\
&=d(v)\sum_{1\leq j\leq s}\big(d_j(v)-1\big)
-\sum_{1\leq i\leq s}d_i(v)\big(d_i(v)-1\big)-\sum_{y\in N_!(v)}d_{c(vy)}(y,N(v)).\\
&=\big(d(v)-d_1(v)\big)\big(d_1(v)-1\big)-\sum_{y\in N_!(v)}d_{c(vy)}(y,N(v))+\sum_{2\leq i\leq s}\big(d(v)-d_i(v)\big)\big(d_i(v)-1\big).
\end{aligned}
\end{equation}

For two disjoint subsets $V_1,V_2\subseteq V(G)$,
the set of colors appearing on the edges between $V_1$ and $V_2$ in $G$
is denoted by $\mathcal{C}(V_1,V_2)$. We say an edge $xy\in E(G)$
is a \emph{rainbow triangle edge} of $v$ if $vxyv$ is a rainbow triangle.
Denote by $RE(v)$ the edge set of rainbow triangle edges of $v$.

First we prove a result on a vertex with maximum monochromatic degree,
i.e., a vertex $v\in V(G)$ with $d^{mon}(v)=\Delta^{mon}(G)$. Recall
that we assume $d_1(v)\geq d_2(v)\geq \ldots \geq d_s(v)$, where
$s=d^c(v)$. Thus, $d_1(v)=d^{mon}(v)$. Also, $N_1(v)$ means the set of all
vertices $u\in N(v)$ such that $c(uv)=1$.

\begin{lem}\label{Lem:B(v)}
Let $G$ be an edge-colored graph.
Then for a vertex $v\in V(G)$ with $d^{mon}(v)=\Delta^{mon}(G)$
we have $B(v)\geq 0$. If $\Delta^{mon}(G)\geq 2$
and $B(v)=0$ then there hold:

(a) $N_!(v)=N(v)\setminus N_1(v)$;

(b) $d^{mon}(u)=\Delta^{mon}(G)$ for all $u\in N_!(v)$; and

(c) if $B_1(v)=0$, then $E[N_1(v),N_!(v)]\subseteq RE(v)$.
\end{lem}

\noindent{\bf Proof of Lemma \ref{Lem:B(v)}.}
If $\Delta^{mon}(G)=1$, then obviously $B(v)=0$. Suppose that $\Delta^{mon}(G)\geq 2$.
Notice that $d_{c(vy)}(y,N(v))\leq \Delta^{mon}(G)-1=d_1(v)-1$ for $y\in N_!(v)$.
From Ineq (\ref{e:B}), we have
\begin{equation}\label{e:B(v)}
\begin{aligned}
B(v)
&\geq\underbrace{\big(d(v)-d_1(v)-|N_!(v)|\big)
\big(d_1(v)-1\big)}_{\textcircled{1}}+
\underbrace{\sum_{2\leq i\leq s}\big(d(v)-d_i(v)\big)\big(d_i(v)-1\big)}_{\textcircled{2}}\\
&\geq 0.
\end{aligned}
\end{equation}

If $B(v)=0$, then
$d_{c(vy)}(y,N(v))=d_1(v)-1$ for all $y\in N_!(v)$
and both $\textcircled{1}$ and $\textcircled{2}$
in Ineq (\ref{e:B(v)}) equal to 0.
Since $\Delta^{mon}(G)=d^{mon}(v)=d_1(v)\geq 2$,
we have $d(v)-d_1(v)-|N_!(v)|=0$ and $d_i(v)=1$ for $i\in [2,s]$.
Therefore, both (a) and (b) hold.

Note that $d(v)-d^c(v)=d_1(v)-1$.
We have $B_1(v)=-\sum_{y\in N_!(v)}d_{c(vy)}(y,N_1(v))=0$.
Hence, $c(vy)\notin \mathcal{C}(N_1(v),N_!(v))$
for $y\in N_!(v)$.
Since $G$ is edge-minimal and $d_1(v)\geq 2$,
we have $1\notin \mathcal{C}(N_1(v),N_!(v))$; indeed,
if $1\in \mathcal{C}(N_1(v),N_!(v))$, then
there exists a monochromatic path of length 3 as $d_1(v)\geq 2$.
Furthermore, for any edge $xy\in E[N_1(v),N_!(v)]$,
$xy$ is a rainbow triangle edge of $v$.
Hence $E[N_1(v),N_!(v)]\subseteq RE(v)$. This proves (c).
The proof is complete.
$\hfill\blacksquare$

\begin{lem}\label{Lem:Bk}
Let $k\geq 2$ be a positive integer and $G$
be a graph on $n\geq 3k-2$ vertices.
If $\delta(G)\geq \frac{n+k-1}{2}$
then $G$ contains a $B_k$.
\end{lem}

\noindent{\bf Proof of Lemma \ref{Lem:Bk}.}
Suppose to the contrary that there is no
$B_k$ in $G$.
If there exists a vertex $v\in V(G)$ with $d(v)\geq \frac{n+k}{2}$,
then $|N(v)\cap N(u)|\geq \lceil\frac{n+k}{2}\rceil+
\lceil \frac{n+k-1}{2}\rceil-n\geq k$
where $u\in N(v)$. In this case, there is a $B_k$.
Suppose $d(v)=\frac{n+k-1}{2}$ for all $v\in V(G)$.
Then for any edge $uv\in E(G)$,
we have
$$\frac{n+k-1}{2}+\frac{n+k-1}{2}-n \leq |N(v)\cap N(u)|\leq k-1.$$
Hence $|N(v)\cap N(u)|=k-1\geq 1$ and $N(u)\cup N(v)=n$.
It follows that $V(G)\setminus N(v)\subseteq N(u)$.
Since each vertex in $V(G)$ is of degree $\frac{n+k-1}{2}$,
we have $|V(G)\setminus N(v)|=\frac{n-k+1}{2}$.
Similarly, for $w\in N(v)\cap N(u)$, we also have $V(G)\setminus N(v)\subseteq N(w)$.
Hence, $V(G)\setminus N(v)\subseteq N(u)\cap N(w)$.
Then $|N(u)\cap N(w)|\geq \frac{n-k+1}{2}$.
Since $\frac{n+k-1}{2}$ is an integer,
$k-1\geq \frac{n-k+1}{2}$, and so $n\leq 3k-3$,
a contradiction. This proves Lemma \ref{Lem:Bk}.
$\hfill\blacksquare$

\vspace{3mm}
Now we are ready to prove Theorem \ref{Thm1:Bk}.

\vspace{3mm}
\noindent{\bf Proof of Theorem \ref{Thm1:Bk}.}
We prove the theorem by contradiction.
Let $G$ be a counterexample such that $e(G)$
is as small as possible.
We use $d_i$ instead of $d_i(v)$ in the following.
Let $v$ be a vertex with $d^{mon}(v)=\Delta^{mon}(G)$.
By Lemma \ref{Lem:Bk}, we can assume $\Delta^{mon}(G)\geq 2$.

If there exists an integer $i \in[1, s]$
such that
$rt(v,N_{i}(v))\geq (k-1)d_{i}+1$,
then there exists a vertex $x_{0} \in N_{i}(v)$ satisfying
$r t(v,x_{0}) \geq k$, a contradiction.
Thus, $$rt(v,N_{i}(v))\leq (k-1)d_{i}\leq \sum_{x\in N_{i}(v)}(d^c(x)+d^c(v)-n)$$ (where
the second inequality holds as the condition that $d^c(x)+d^c(v)\geq n+k-1$).
It follows that $B_i(v)\leq 0$
for all $i \in[1, s]$ by Lemma \ref{lem:rt(v,N_i(v))}.
Therefore,
$0\leq B(v) =\sum_{1\leq i\leq s} B_i(v)\leq 0$.
That is, $B(v)=B_i(v)=0$ for $i\in [1,s]$.
By Lemma \ref{Lem:B(v)} we have
$rt(v,N_i(v))= (k-1)d_{i}$ for all $i \in[1, s]$
and   $d_i=1$ for $i\in [2,s]$.
Since $G$ contains no $k$ rainbow triangles sharing
one common edge,
$rt(v,x)=k-1$ for all $x\in N(v)$.

\vspace{2mm}
\noindent{\bf Claim.} There exists a vertex $x\in N_!(v)$
such that $d^c_{G[N[v]]}(x)\leq k$.

\begin{proof}
Choose $x\in N_!(v)$. If all edges in $G[N_!(v)]$ are rainbow edges of $v$,
then as $E[N_1(v),N_!(v)]\subseteq RE(v)$ (recall Lemma \ref{Lem:B(v)}(c)),
$d^c_{G[N[v]]}(x)\leq d_{G[N[v]]}(x)\leq rt(v,x)+1\leq k$
for $x\in N_!(v)$. Now assume that
there exists a vertex $u'\in N_!(v)$ such that $c(xu')=c(vx)$.
Since $x\in N_!(v)$, by (b) of Lemma \ref{Lem:B(v)},
$d^{mon}(x)=\Delta^{mon}(G)$.

By a) of Lemma \ref{Lem:B(v)},
$N_!(x)=N(x)\setminus N_{c(vx)}(x)$ and $c(vx)$
is the maximum monochromatic color of $x$.
For all $u\in N_!(v)\cap N(x)$, we have
$u\in N_{c(xv)}(x)$ or $u\in N_!{x}$.
By c) of Lemma \ref{Lem:B(v)},
$E[N_{c(xv)}(x),N_!(x)]\subseteq RE(x)$.
If $u\in N_!(x)$, then $xvux$ is a rainbow triangle,
and hence $xu$ is a rainbow triangle edge of $v$.
Therefore, for all $u\in N_!(v)\cap N(x)$, we
have $c(xu)=c(vx)$ or $xu\in RE(v)$.
From c) of Lemma \ref{Lem:B(v)}, we have $E[x,N_1(v)]\subseteq RE(v)$.
Hence, $$d^c_{G[N[v]]}(x)\leq d^c_{N_![v]}(x)+|\mathcal{C}(x,N_1(v))|\leq rt(v,x)+1=k.$$
\end{proof}

Then, we infer
$$\begin{aligned}
\frac{n+k-1}{2}\leq d^c(x)&\leq d^c_{G[N[v]]}(x)+n-\big(\Delta^{mon}(G)+d^c(v)\big)\\
&=k+n-\Delta^{mon}(G)-d^c(v)\\
&\leq \frac{n+k+1}{2}-\Delta^{mon}(G),\\
\end{aligned}$$
that is, $\Delta^{mon}(G)= 1$,
a contradiction.
$\hfill\blacksquare$

\subsection{Edge-colored friendship graphs}
The aim of this subsection is to prove Theorem \ref{Thm2:Fk}.

A \emph{matching} of a graph consists of some vertex-disjoint edges.
The \emph{matching number} of a graph $G$ is
the maximum number of pairwise disjoint edges in $G$, denoted by $\alpha'(G)$.
Erd\H{o}s and Gallai \cite{EG59} determined Tur\'{a}n number of a matching with
given size.

A \emph{covering} of a graph $G$ is
a subset $K$ of $V(G)$ such that every edge of
$G$ has at least one end vertex in $K$.
A covering $K^*$ is a \emph{minimum covering} if $G$ has no
covering $K$ with $|K|<|K^*|$.
The number of vertices in a minimum covering of $G$
is called the \emph{covering number} of $G$, and is denoted by $\beta(G)$.

\begin{lem}\label{Lem:beta(G)-com}
For a graph $G$ on $n$ vertices, we have $\beta(G)\leq n-1$.
Furthermore, $G$ is complete if and only if $\beta(G)= n-1$.
\end{lem}

\noindent{\bf Proof of Lemma \ref{Lem:beta(G)-com}.}
The sufficiency
of condition  is clear. We establish
its necessity by the contradiction.
Suppose that $G$ is not complete.
Then there exist two vertices $u,v\in V(G)$ such that
$uv\notin E(G)$. Hence $V(G)\setminus\{u,v\}$ is a covering of $G$.
Thus, $\beta(G)\leq n-2$, a contradiction.
$\hfill\blacksquare$

\vspace{4mm}
Now we will prove a useful bound on $\alpha'(G)$ and
$\beta(G)$, which is partly inspirited by the
proof of Tur\'an numbers of matchings by
Erd\H{o}s and Gallai \cite{EG59}. Following Berge \cite[pp.~176]{B58} (see
also Erd\H{o}s and Gallai \cite[pp.~354]{EG59}),
for a graph $G$ with $\alpha'(G)=k$ and $n>2k$,
choose $k$ independent edges
and call them $\alpha'$-edges and the remaining edges $\gamma$-edges.
Add a vertex $x$ to $V(G)$ and connect $x$ with
every vertex which is not incident with $\alpha'$-edges
(i.e., any vertex not in the edges of the matching of size $k$) by an edge.
For all the new edges (i.e., the edges we add) and the old $\alpha'$-edges
(i.e., the matching), we call them $\alpha$-edges.
A path is called \emph{alternating} if
its edges are alternately $\alpha$-edges and $\gamma$-edges.
For a vertex, we say that it is a $\gamma$-vertex if it
is reachable from $x$ by an alternating path which only ends with $\gamma$-edges.
Apparently, $x$ is a $\gamma$-vertex.

Let $V_0$ be the set of all $\gamma$-vertices of $G$,
and $V_1,\cdots,V_p$ be the components of $G$ obtained from $G$ by deleting $V_0$.
There is a well-known fact (\cite[pp.~169-170]{B50,B58},
\cite[pp.~141-142]{G50})
on the relationship among
$\alpha$-edges, $\gamma$-vertices and $V_i$.

\begin{fact}\label{Fact:alpha-edge}
Each $\alpha$-edge incident to a $\gamma$-vertex
is also incident to some odd $V_i$
(i.e., $|V_i|$ is odd)
and
each odd $V_i$ has exactly one $\alpha$-edge incident to a $\gamma$-vertex.
\end{fact}

It follows that each $\alpha
$-edge is either incident to
a $\gamma$-vertex and an odd $V_i$,
or is an edge of some $G[V_i]$
(see Figure \ref{fig1}, the edges colored blue are $\alpha'$-edges).
\begin{figure}[htbp!]
  \centering
\scalebox{1}{\includegraphics[width=3.9in,height=2.0in]{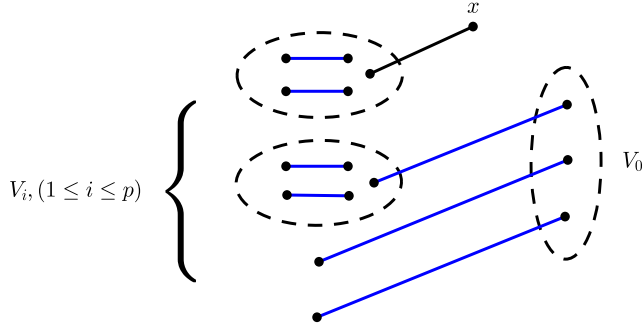}}\\
\captionsetup{font={scriptsize}}
\caption{A partition $(V_0,V_1,\cdots,V_p)$ of a graph $G$}\label{fig2}
\end{figure}
Furthermore, we can assert that $G[V_i]$ contains exactly
$\lfloor \frac{|V_i|}{2}\rfloor$ $\alpha'$-edges.
If $\alpha'(G)<\lfloor \frac{|V_i|}{2}\rfloor$,
then there exist at least two vertices  in $V_i$
incident to $x$, a contradiction to Fact \ref{Fact:alpha-edge}.
Note that each vertex in $V_0$ is incident with an $\alpha'$-edge;
otherwise the vertex is connected with $x$ by an $\alpha$-edge, a contradiction.
The following equation on $\alpha'(G)$
can be deduced implicitly from  Erd\H{o}s and Gallai \cite[pp.~355]{EG59}.

\begin{equation}\label{e:alpha'}
\alpha'(G)= |V_0|+\sum_{1\leq i\leq p}\left\lfloor\frac{|V_i|}{2}\right\rfloor.
\end{equation}

Since $\beta(V_i)\leq |V_i|-1$ for $1\leq i\leq p$,
by Lemma \ref{Lem:beta(G)-com},
we have the following lemma.

\begin{lem}\label{Lem:alpha}
Let $G$ be a graph on $n>2\alpha'(G)$ vertices and $\{V_i:0\leq i\leq p\}$ be the partition
as remarked above.
Then
\begin{equation}\label{e:beta}
\beta(G)\leq |V_0|+\sum_{1\leq i\leq p}(|V_i|-1)=n-p\leq 2\alpha'(G)-|V_0|.
\end{equation}
\end{lem}

\begin{lem}\label{Lem:beta(G)}
Let $G$ be a graph on $n\geq 2\alpha'(G)+2$ vertices
and $\{V_i:0\leq i\leq p\}$ be the partition
as remarked above.
Then $\beta(G)\leq 2\alpha'(G)-1$
and $V_0$ is not empty.
Furthermore,
if $\beta(G)=2\alpha'(G)-1$,
then

(1) $V_i$ is complete and odd for $1\leq i\leq p$;

(2) $\beta(G)=n-p$
and the covering of $G$ consists of all vertices of $V_0$ and $|V_i|-1$ vertices of $V_i$ for $1\leq i\leq p$.
\end{lem}

\noindent{\bf Proof of Lemma \ref{Lem:beta(G)}.}
Since $n\geq 2\alpha'(G)+2$,
$G$ is not complete.
If $V_0=\emptyset$,
then $G$ contains exactly
$\lfloor \frac{n}{2}\rfloor$ $\alpha'$-edges,
which implies that $n\leq 2\alpha'(G)+1$, a contradiction.
Hence $|V_0|\geq 1$.
$\beta(G)\leq 2\alpha'(G)-1$ follows from Ineq (\ref{e:beta}).

If $\beta(G)=2\alpha'(G)-1$,
then all inequalities become equalities in Ineq (\ref{e:beta}).
Hence $\beta(V_i)=|V_i|-1$ for all $1\leq i\leq p$.
From Lemma \ref{Lem:beta(G)-com}, $G[V_i]$ is complete.
$\hfill\blacksquare$

\vspace{4mm}
Recall $RE(v)$ is the edge set of rainbow triangle edges of $v$.
Let $G_{\triangle}(v)$ denote the subgraph induced by the edge set $RE(v)$.
Let $C_{\triangle}(v)$ denote a minimum covering of $G_{\triangle}(v)$.
Thus $\beta(G_{\triangle}(v))=|C_{\triangle}(v)|$.
Now we are ready to prove Theorem \ref{Thm2:Fk}.

\vspace{4mm}
\noindent{\bf Proof of Theorem \ref{Thm2:Fk}.}
We prove the theorem by contradiction.
Let $G$ be a counterexample with the smallest order $n$,
and then $e(G)$ is as small as possible.
Choose $v\in V(G)$ such that $d^{mon}(v)=\Delta^{mon}(G)$.
Then $\alpha'(G_{\triangle}(v))\leq k-1$.
Since $n\geq 2k+9$,
by Corollary \ref{Cor:rt(v)} and Lemma \ref{Lem:B(v)}
we have $e(G_{\triangle}(v))=rt(v)\geq \frac{4k^2-9}{2}>\binom{2k-1}{2}$.
Hence $|V(G_{\triangle}(v))|\geq 2k\geq 2\alpha'(G_{\triangle}(v))+2$.
Then $\beta(G_{\triangle}(v))\leq 2k-3$
by Lemma \ref{Lem:beta(G)}.

Let $\alpha$ be a color satisfying $d^{mon}(v)=d_{\alpha}(v)$.
Now we define two disjoint vertex subsets of $N(v)$.
If $\Delta^{mon}(G)\geq2$,
let $U(v)$ be a maximum rainbow neighborhood of $v$
 in $N(v)\setminus (N_{\alpha}(v)\cup C_{\triangle}(v))$ and
$W(v)= N_{\alpha}(v)\setminus C_{\triangle}(v)$.
If $\Delta^{mon}(G)=1$,
let $U(v)$ be a maximum rainbow neighborhood of $v$
 in $N(v)\setminus C_{\triangle}(v)$ and
$W(v)= \emptyset$.
Set $X(v)=U(v)\cup W(v)$, $T(v)=X(v)\cup C_{\triangle}(v)$ and $T[v]=T(v)\cup \{v\}$.
Thus, $X(v)=(N(v)\cup N_{\alpha}(v))\setminus C_{\triangle}(v)=N(v)\setminus C_{\triangle}(v)$.
Apparently, $|U(v)|\geq d^c(v)-\beta(G_{\triangle}(v))-1$ and
\begin{equation}\label{e:T[v]}
|T[v]|\geq |N_{\alpha}(v)|+d^c(v)-1+1 \geq d^{mon}(v)+d^c(v).
\end{equation}
Define $q(u)=|\mathcal{C}(u,C_{\triangle}(v))\setminus \mathcal{C}(u,X[v])|$ for $u\in X(v)$.

According to the definition of $C_{\triangle}(v)$,
there is no rainbow triangle edge of $v$ in $G[\{v\}\cup X(v)]$.
For $u_1,u_2\in X(v)$, if $u_1u_2\in E(G)$,
then we have
\begin{equation}\label{c(u_1u_2)}
c(u_1u_2)\in \{c(vu_1),c(vu_2)\}.
\end{equation}

Let $Y\subseteq X(v)$
and $D[Y]$ be the oriented graph of $G[Y]$ defined as follows:
let $V(D)=Y$ and
orient each edge $u_1u_2$ in $G[Y]$ by $\overrightarrow{u_1u_2}$
if $c(u_1u_2)=c(vu_2)$.
Then for $u\in Y$, all in-arcs from $u$ are assigned the color $c(vu)$,
and all out-arcs from $u$ are assigned pairwise distinct colors
which are different from $c(vu)$ as $G$ is edge-minimal.
Hence $d^{mon}_{G[Y]}(u)=d^-_{D[Y]}(u)$ and
$d^{c}_{G[Y]}(u)=d^+_{D[Y]}(u)$.
Since $d^-_{D[Y]}(u)\leq \Delta^{mon}(G)-1$
for any $u\in Y$,
there exists a vertex $u'\in Y$
such that
\begin{equation}\label{e:*}
d^{c}_{G[Y]}(u')=d^+_{D[Y]}(u')\leq \Delta^{mon}(G)-1.
\end{equation}

In the following, we prove a useful claim on a vertex $v\in V(G)$
with $d^{mon}(v)=\Delta^{mon}(G)$.

\vspace{2mm}
\noindent{\bf Claim.} $q(u)=\beta(G_{\triangle}(v))=2k-3$ for all $u\in X(v)$.

\begin{proof}
Apparently, we have
 \begin{equation}\label{e:d^c(u)}
\begin{aligned}
d^c(u)&\leq d^c_{T[v]}(u)+n-\big|T[v]\big|.
\end{aligned}
\end{equation}

If $\Delta^{mon}(G)=1$, then $G$ is properly-colored.
Therefore, $G[X(v)]$ consists of isolated vertices by (\ref{c(u_1u_2)}).
Thus, for $u\in X(v)$,
$d^c_{T[v]}(u)=|c(uv)\cup \mathcal{C}(\{u\},C_{\triangle}(v))|\leq 1+q(u)$.
Hence by Ineqs (\ref{e:T[v]}) and (\ref{e:d^c(u)}),
we have
$d^c(u)+d^c(v)-n\leq q(u)\leq \beta(G_{\triangle}(v))\leq 2k-3.$
Now we have $q(u)=\big|C_{\triangle}(v)\big|=2k-3$ for all $u\in X(v)$.

Assume that $\Delta^{mon}(G)\geq 2$.
Then for $w \in W(v)$ and $u \in U(v)$,
$c(wu)=c(vu)$ if $wu\in E(G)$
(otherwise, there is a monochromatic
$P_3$, a contradiction).
Thus for $u\in U(v)$,
we have $d^c_{T[v]}(u)\leq d^+_{D[U(v)]}(u)+1+q(u)$.
By Ineqs (\ref{e:T[v]}) and (\ref{e:d^c(u)}),
we have
\begin{equation}\label{e:d^+(u)}
\begin{aligned}
d^{+}_{D[U(v)]}(u)&\geq d^c(u)+\big|T[v]\big|-n-q(u)-1\\
&\geq \Delta^{mon}(G)+2k-4-q(u).
\end{aligned}
\end{equation}

Since $q(u)\leq \beta(G_{\triangle}(v))\leq 2k-3$, we have $d^{+}_{D[U(v)]}(u)\geq \Delta^{mon}(G)-1$.
Since $d^-_{D[U(v)]}(u)\leq \Delta^{mon}(G)-1$,
$d^{+}_{D[U(v)]}(u)=d^{-}_{D[U(v)]}(u)=\Delta^{mon}(G)-1$,
which gives that $q(u)=2k-3$ for all $u\in U(v)$
and all inequalities in Ineq (\ref{e:d^+(u)}) become equalities.

Meanwhile, we have $E[W(v),U(v)]=\emptyset$
as $c(wu)=c(vu)$ if $wu\in E(G)$.
Hence for $w\in W(v)$,
\begin{equation}\label{e:d^c(w)}
\begin{aligned}
\frac{n+2k-3}{2}\leq d^c(w)&\leq d^c_{T[v]}(w)+n-\big|T[v]\big|\\
&\leq \Delta^{mon}(G)+q(w)+n-\Delta^{mon}(G)-d^c(v)\\
&\leq\frac{n-2k+3}{2}+q(w).
\end{aligned}
\end{equation}
Then $q(w)=2k-3$ for all $w\in W(v)$.
\end{proof}

Let $\{V_i:0\leq i\leq p\}$ be the partition of $G_{\triangle}(v)$
as remarked in Lemma \ref{Lem:alpha}.
Since $G_{\triangle}(v)$  is non-complete, by Claim,
$2k-3=\beta(G_{\triangle}(v))\leq 2\alpha'(G_{\triangle}(v))-|V_0|\leq 2k-3$,
which guarantees that $\alpha'(G_{\triangle}(v))=k-1$
and $|V_0|=1$. Since $|V(G_{\triangle}(v))|\geq 2k$, by Lemma \ref{Lem:alpha},
we have $$p\geq |G_{\triangle}(v)|+|V_0|-2\alpha'(G_{\triangle}(v))\geq 2k+1-2(k-1)=3.$$

We first consider the case of $V_0\subsetneq C_{\triangle}(v)$.
By Lemma \ref{Lem:beta(G)}(2),
for $x\in V_i\cap C_{\triangle}(v),1\leq i\leq p$
there exists exactly one vertex $u'\in X(v)$
such that $xu'\in RE(v)$, as $|V_i\setminus C_{\triangle}(v)|=1$.
For any $u\in X(v)\setminus V(G_{\triangle})$ and $x\in C_{\triangle}(v)$,
the fact $xu\notin RE(v)$ gives that $c(xu)\in \{c(vu),c(vx)\}$.
Since $q(u)=|\mathcal{C}(u,C_{\triangle}(v))\setminus \mathcal{C}(u,X[v])|=|C_{\triangle}(v)|=\beta_{\triangle}$,
$c(xu)=c(vx)$.
Since $q(u)=|C_{\triangle}(v)|=\beta_{\triangle}$, by Claim,
each vertex in $X(v)$ is incident to all vertices in $C_{\triangle}(v)$.
For $x\in V_i\cap C_{\triangle}(v)$, there is only one vertex $u'\in X(v)$ such that $c(xu')\neq c(vx)$.
Hence $d^{mon}(x)\geq d_{c(vx)}(x)\geq |X(v)\cup \{v\}-\{u'\}|$.
Since $|X(v)\cup \{v\}|=|T[v]|-\beta_{\triangle}$,
we have the following inequality:
$$\begin{aligned}
d^{mon}(x)&\geq |(X(v)\cup \{v\})-\{u'\}|\\
&\geq d^c(v)+\Delta^{mon}(G)-\beta(G_{\triangle}(v))-1\\
&\geq \frac{n+2k-3}{2}+\Delta^{mon}(G)-(2k-3)-1\\
&\geq \Delta^{mon}(G)+1
\end{aligned}$$
as $n\geq 2k+1$. This is a contradiction.

Next we consider the case that $C_{\triangle}(v)=V_0$.
The goal of the coming part is to prove that $d^{mon}(u)=\Delta^{mon}(G)$
for any $u\in X(v)$.

Suppose that there exists a vertex $u'\in X(v)$ with
$d^c_{G[X(v)]}(u')\leq \Delta^{mon}(G)-2$. Recall that
$T(v)=X(v)\cup C_{\triangle}(v)$, and $X(v)$ and $C_{\triangle}(v)$
are disjoint. Then by Claim and (\ref{e:T[v]}), we have
$$\begin{aligned}
&\frac{n+2k-3}{2}\leq d^c(u')\\
&\leq d^c_{G[X(v)]}(u')+d^c_{G[\{u'\}\cup C_{\triangle}(v)]}(u')+n-|T[v]|+|\{v\}|\\
&=d^c_{G[X(v)]}(u')+q(u')+n-|T(v)|\\
&\leq \Delta^{mon}(G)-2+2k-3+n-(d^{mon}(v)+d^c(v)-1)\\
&\leq 2k-4+n-\frac{n+2k-3}{2},
\end{aligned}$$
a contradiction. Thus, for any vertex $u\in X(v)$, we have
$d^c_{G[X(v)]}(u)\geq\Delta^{mon}(G)-1$. Set $Y=X(v)$. Recall the rule of
the construction of $D[Y]$ (above the proof of Claim). We have $d^-_{D[Y]}(u)\leq \Delta^{mon}(G)-1$
for any $u\in Y$. Since $\sum_{u\in Y}d^-_{D[Y]}(u)=\sum_{u\in Y}d^+_{D[Y]}(u)$,
each vertex $u\in X(v)$ satisfies that
$d^c_{G[X(v)]}(u)= \Delta^{mon}(G)-1$.
Hence $d^{mon}(u)= d^-_{D[X(v)]}(u)+1=\Delta^{mon}(G)$.

Since $C_{\triangle}(v)=V_0$ and $|V_0|=1$,
we have $q(u)=2k-3=1$, and hence $k=2$.
Note that $|G_{\triangle}(v)-C_{\triangle}(v)|=\sum_{i=1}^p|V_i|\geq p\geq 3$.
To avoid the graph consisting of two rainbow
triangles sharing one common vertex in $G$,
we have $|RE(u)|\leq 1$
for $u\in G_{\triangle}(v)-C_{\triangle}(v)$.
Since $u$ is also a vertex with $d^{mon}(u)=\Delta^{mon}(G)$, we have $|RE(u)|\geq 2$, a contradiction.
$\hfill\blacksquare$

\section{Concluding remarks}\label{Sec:5}
One may wonder the sharpness of Theorems \ref{Thm1:Bk}
and \ref{Thm2:Fk}. For Theorem \ref{Thm1:Bk},
by Example 1, we know when $k\geq \frac{n}{3}+1$ (in
this case, the subgraph $B_k$ is with order $\Theta(n)$),
the color degree guaranteeing a properly colored $B_k$
should be larger than $\frac{n+k-1}{2}$.

For uncolored friendship subgraphs, we first prove the following
result.
\begin{prop}\label{Prop:F}
Let $k\geq 2$ be a positive integer and
$G$ be a graph on $n\geq 3k-1$ vertices.
If $\delta(G)\geq \frac{n+k-1}{2}$
then $G$ contains a $F_k$.
\end{prop}

\vspace{4mm}
\noindent{\bf Proof of Proposition \ref{Prop:F}.}
We proceed the proof by induction on $k$.
The basic case $k=2$ is easily derived from
Theorem \ref{Thm2:Fk}.
Let $k\geq 3$ and suppose the result holds for $k-1$.
Suppose to the contrary that there is no
$F_k$ in $G$.
Let $v$ be the center of a $F_{k-1}$ and $S=V(F_k-v)$.
Since $\delta(G)\geq \frac{n+k-1}{2}$,
each edge is contained in at least $k-1$ triangles.
As  $n\geq 3k-1$,
there exist $k$ vertices in $N(v)\setminus S$.
According to pigeonhole principle,
there exists a $F_k$ in $N[v]$.
$\hfill\blacksquare$

As we have already mentioned in the introduction, as corollaries
of the result of Erd\H{o}s et al. \cite{EFGG} on $k$-fans and Erd\H{o}s' conjecture
on books, respectively, the following results hold.
\begin{prop}\label{Prop:F-2}
Let $k\geq 2$ be a positive integer and
$G$ be a graph on $n\geq 50k^2$ vertices.
If $\delta(G)\geq \frac{n+1}{2}$
then $G$ contains a $F_k$.
\end{prop}

\begin{prop}\label{Prop:F-2}
Let $k\geq 2$ be a positive integer and
$G$ be a graph on $n\geq 6k$ vertices.
If $\delta(G)\geq \frac{n+1}{2}$
then $G$ contains a $B_k$.
\end{prop}

So by reasoning the above results, Theorems \ref{Thm1:Bk} and \ref{Thm2:Fk} should
be at least asymptotically tight when $k=o(n)$.

For Theorem \ref{Thm2:Fk}, the corresponding color degree condition may be acceptable
when one considers a nearly spanning $F_k$. One evidence is that, if we
consider the color degree condition for the spanning $F_k$ (that is, $k=\frac{n-1}{2}$),
then the sufficient condition is that $\delta^c\geq n-1$ which
equals to $\frac{n+2k}{2}+O(1)$ (see the following).

\begin{fact}
Let $G$ be an edge-colored graph on $n$ vertices
where $n$ is odd. If $\delta^c\geq n-1$ then $G$
contains a properly-colored $F_{\frac{n-1}{2}}$.
\end{fact}

\section*{Acknowledgment}
The results presented here were reported by the first author in
a workshop to congratulate the birthday of Prof. Jinjiang Yuan. The authors are very grateful to
Prof. Jinjiang Yuan for his encouragement and comments.

\end{document}